\documentclass[a4paper, 12pt]{amsproc}
\allowdisplaybreaks

\usepackage{fullpage}
\usepackage{amsmath, amsthm, amssymb, mathtools, mathrsfs}
\DeclarePairedDelimiter\abs{\lvert}{\rvert}
\usepackage{bm}

\usepackage{chessfss}
\usepackage{mathtools}
\usepackage{tikz}
\usetikzlibrary{intersections, calc, arrows.meta}

\usepackage{pgfplots}

\usepackage{graphicx}
\usepackage{here}
\usepackage{time}
\usepackage{bbm}

\usepackage{xcolor}
\usepackage{hyperref}
\usepackage[capitalize,nameinlink,noabbrev,nosort]{cleveref}
\hypersetup{
	colorlinks=true, 
	linkcolor=brown, 
	citecolor=brown, 
	filecolor=brown, 
	urlcolor=brown, 
}

\makeatletter
\@namedef{subjclassname@2020}{%
 \textup{2020} Mathematics Subject Classification}
\makeatother

\newtheorem{theoremcounter}{Theorem Counter}[section]

\theoremstyle{definition}

\newtheorem{example}[theoremcounter]{Example}

\theoremstyle{plain}
\newtheorem{lemma}[theoremcounter]{Lemma}
\newtheorem{proposition}[theoremcounter]{Proposition}
\newtheorem{corollary}[theoremcounter]{Corollary}

\newtheorem{theorem}[theoremcounter]{Theorem}

\numberwithin{equation}{section}

\newcommand{\N}{\mathbb{N}}
\newcommand{\Z}{\mathbb{Z}}
\newcommand{\Hom}{\operatorname{Hom}}
\newcommand{\End}{\operatorname{End}}
\newcommand{\Proj}{\operatorname{Proj}}
\newcommand{\Cyc}{\operatorname{Cyc}}

\begin{document}

\title[]{Cyclic-quasi-injective for Finite Abelian Groups} 

\author{Yusuke Fujiyoshi} 
\address{Joint Graduate School of Mathematics for Innovation, Kyushu University,
Motooka 744, Nishi-ku, Fukuoka 819-0395, Japan}
\email{fujiyoshi.yusuke.671@s.kyushu-u.ac.jp}




\maketitle

\begin{abstract}
We investigate the conditions for a finite abelian group $G$ under which any
cyclic subgroup $H$ and any group homomorphism $f \in \Hom(H,G)$ can be extended to
an endomorphism $F \in \End(G)$. As a result, we provide necessary and sufficient conditions for such a group $G$ and we compute the number of cyclic subgroups possessing non-extendable homomorphisms.
In addition, we demonstrate that the number of cyclic subgroups that do not satisfy the conditions corresponds to the sum of the maximum jumps in the associated permutations given by $\sum_{\sigma \in S_{n}} \max_{1 \leq i \leq n} \{\sigma(i) - i\}$.

\end{abstract}

\section{Introduction}

A group $G$ is said to be \textit{quasi-injective} if there exists an endomorphism $F \in \End(G)$ such that $F|_H = f$ for every subgroup $H$ of $G$ and every homomorphism $f \in \Hom(H,G)$.
Kil'p~\cite{Kil} (see also \cite{Fuc}) proved that a finite abelian group $G$ is a quasi-injective  if and only if the $p$-component of $G$ is a homocyclic group for any prime $p$. A homocyclic group is a direct product of cyclic groups of the same order, where this order is the power of a prime $p$. It takes the form $(\Z \slash p^{r}\Z)^{l}$ for a prime $p$ and $r,l \in \N$. For finite non-abelian groups, Bertholf and Walls \cite{BW} completely characterized quasi-injective groups. They showed that such groups are supersoluble, metabelian, and every subgroup is a $T$-group (i.e., every subnormal subgroup is normal). Extending this line of research to infinite groups, Tomkinson \cite{Tom} investigated quasi-injective groups in the soluble and locally finite cases. He demonstrated that, analogous to the finite case, such groups are $T$-groups.

In this paper, we consider a weakened version of the quasi-injective group condition, which we call a \textit{cyclic-quasi-injective group}. A group $G$ is said to be \textit{cyclic-quasi-injective} if there exists an endomorphism $F \in \End(G)$ such that $F|_H = f$ for every cyclic subgroup $H$ of $G$ and every homomorphism $f \in \Hom(H, G)$. Based on this, we define $X(G)$ as the set of cyclic subgroups $H$ of $G$ such that there exists a homomorphism $f \in \Hom(H, G)$ that satisfies $F|_H \neq f$ for any endomorphism $F \in \End(G)$. Especially, $X(G) = \emptyset $ if and only if $G$ is a cyclic-quasi-injective group.
We will establish the conditions under which a finite abelian group $G$ becomes a cyclic-quasi-injective group by determining the number of elements in $X(G)$.

The main results of this paper are given by \cref{thm:1.1} and \cref{thm:1.3}.

\begin{theorem} \label{thm:1.1}
Let $G$ be a finite abelian group. Then $G$ is a cyclic-quasi-injective group if and only if the $p$-component of $G$ is a homocyclic group for any prime $p$.
\end{theorem}

We will introduce several definitions to carry out the proof. Here, $p$ denotes a prime number throughout this paper. For $ r \in \N $, let $ Z(p^r) = \Z \slash p^{r} \Z $. Furthermore, for $ x \in Z(p^{r}) $, we define $v_{p}(x)$ as the largest non-negative integer $k$ such that $x$ can be expressed as $x =  mp^k$ for some integer $m$ that is not divisible by $p$. We set $v_{p}(0) = r$.
Furthermore, let $ G = \bigoplus_{i=1}^{n} Z(p^{m_{i}}) $. Let $ H $ and $ H' $ be cyclic subgroups of $ G $, and let $ h_{i} $ and $ h_{i}' $ denote the $i$-th components of the generators of $ H $ and $ H' $, respectively, where $ h_{i}, h_{i}' \in Z(p^{m_{i}}) $. We define an equivalence relation by
\[
H {\sim} H' \ \overset{\text{def}}{\Longleftrightarrow}\  \forall i \in \{ 1,2,\cdots, n\},\ v_{p}(h_{i}) = v_{p}(h'_{i}).
\]
From \Cref{lemma:2.5}, this equivalence relation does not depend on the choice of representatives. We consider $X(G) \slash {\sim}$ to investigate $X(G)$.
\vskip\baselineskip

\begin{example}
Let $G = Z(p^{2}) \oplus Z(p^{5})$. Then $\# X(G) = 3p + 1$ and $\# X(G) \slash {\sim} = 7$ where $\#$ denotes the number of elements in a set.
\end{example}

Using $1 \leq t \leq p-1$, we have
\begin{align*}
X(G) &= \{ \langle (tp,p) \rangle,\ \langle (tp,p^2) \rangle,\ \langle (tp,p^3) \rangle,\ \langle (0,p) \rangle,\ \langle (0,p^2) \rangle,\ \langle (0,p^3) \rangle,\ \langle (0,p^4) \rangle \}, \\
X(G)\slash {\sim} &= \{ [\langle (p,p) \rangle],\ [\langle (p,p^2) \rangle],\ [\langle (p,p^3) \rangle],\ [\langle (0,p) \rangle],\ [\langle (0,p^2) \rangle],\ [\langle (0,p^3) \rangle],\ [\langle (0,p^4) \rangle] \}.
\end{align*}
Therefore, we obtain
\[
\# X(G) = 3(p-1) + 4 = 3p+1,\ \ \# X(G) \slash {\sim} = 7.
\]

For a finite abelian group $G$, the explicit formulas for $\# X(G)$ and $\# X(G) \slash {\sim}$ are given below.
\vskip\baselineskip

\begin{theorem} \label{thm:1.3}
Let $G = \displaystyle \bigoplus_{i=1}^{n} Z(p^{m_{i}})^{\lambda_{i}} \ (1 \leq m_{i} < m_{i+1},\lambda_{i} \in \N)$ and let $\abs{\lambda}_{j} = \sum_{i=1}^{j} \lambda_{i}$. Then we obtain
\[
\# X(G) \slash {\sim} = S_{1}(G) + S_{2}(G) + S_{3}(G)
\]
where
\begin{align*}
S_{1}(G) &= \sum_{1 \leq k < m_{1}} k^{\lambda_{1}} \bigg( (k+1)^{\abs{\lambda}_{n} - \lambda_{1}} - k^{\abs{\lambda}_{n} - \lambda_{1}} \bigg),\\
S_{2}(G) &= \sum_{i=1}^{n-1} \sum_{m_{i} \leq k < m_{i+1}} m_{i}^{\lambda_{i}} 
\bigg( \prod_{j=1}^{i-1} (m_{j}+1)^{\lambda_{j}} \bigg)
\bigg( (k+1)^{\abs{\lambda}_{n} - \abs{\lambda}_{i}} - k^{\abs{\lambda}_{n} - \abs{\lambda}_{i}} \bigg),\\
S_{3}(G) &= \sum_{i=1}^{n-1} \sum_{m_{i} < k < m_{i+1}} 
k^{\lambda_{i+1}} \bigg((m_{i}+1)^{\lambda_{i}} - m_{i}^{\lambda_{i}} \bigg)
\bigg( \prod_{j=1}^{i-1} (m_{j}+1)^{\lambda_{j}} \bigg)
\bigg( (k+1)^{\abs{\lambda}_{n} - \abs{\lambda}_{i+1}} - k^{\abs{\lambda}_{n} - \abs{\lambda}_{i+1}} \bigg).
\end{align*}

Furthermore, we obtain
\[
\# X(G) = T_{1}(G) + T_{2}(G) + T_{3}(G)
\]
where
\begin{align*}
T_{1}(G) &= \sum_{1 \leq k < m_{1}} p^{(k-1)\lambda_{1}} \cdot
\frac{p^{k(\abs{\lambda}_{n}- \lambda_{1})} - p^{(k-1)(\abs{\lambda}_{n}- \lambda_{1})}}{p^{k-1}(p-1)},\\
T_{2}(G) &= \sum_{i=1}^{n-1} \sum_{m_{i} \leq k < m_{i+1}} 
\bigg( \prod_{j=1}^{i-1} p^{m_{j}\lambda_{j}} \bigg) p^{(m_{i}-1)\lambda_{i}} \cdot
\frac{p^{k(\abs{\lambda}_{n}- \abs{\lambda}_{i})} - p^{(k-1)(\abs{\lambda}_{n}- \abs{\lambda}_{i})}}{p^{k-1}(p-1)},\\
T_{3}(G) &= \sum_{i=1}^{n-1} \sum_{m_{i} < k < m_{i+1}} 
\bigg( \prod_{j=1}^{i-1} p^{m_{j}\lambda_{j}} \bigg) 
(p^{m_{i}\lambda_{i}} - p^{(m_{i}-1)\lambda_{i}})
p^{(k-1)\lambda_{i+1}} \cdot
\frac{p^{k(\abs{\lambda}_{n}- \abs{\lambda}_{i+1})} - p^{(k-1)(\abs{\lambda}_{n}- \abs{\lambda}_{i+1})}}{p^{k-1}(p-1)}.
\end{align*}
\end{theorem}

\section{Proof of Theorem 1.3}

In this section, we prove \cref{thm:1.3}. For a finite abelian group $G$ whose order is a power of $p$, we determine the conditions under which a cyclic subgroup has a non-extendable homomorphism.
\\

\begin{proposition} \label{prop:2.1}
Let $G = \displaystyle \bigoplus_{i=1}^{n} Z(p^{m_{i}})^{\lambda_{i}} \ (1 \leq m_{i} < m_{i+1},\lambda_{i} \in \N)$ and let $H, H'$ be cyclic subgroups of $G$. If $H \in X(G)$ and $H {\sim} H'$, then $H' \in X(G)$.
\end{proposition}

\begin{proof}
Let $H = \langle (h_{1},\cdots,h_{n}) \rangle$ be a cyclic subgroup of $G$, and define $\boldsymbol{h} = (h_{1},\cdots,h_{n})$, $\boldsymbol{x} = (x_{1},\cdots,x_{n})$, and $\boldsymbol{e}_{i} = (0,\cdots,0, \underset{i\text{-th}}{1} ,0,\cdots,0) \in G$. Since $G$ is a direct product of cyclic groups, the existence of an extension $F \in \End(G)$ of $f$ is equivalent to identifying $F(\boldsymbol{e}_{i})$ for each $i$.
Assume that an extension $F$ of $f$ exists. Let the order of $\boldsymbol{h}$ be $p^{u}$, and set $\alpha_{i} = v_{p}(h_{i})$ for each $i$. Additionally, suppose that $f(\boldsymbol{h}) = \boldsymbol{x}$.
Since $p^{u} \boldsymbol{h} = \boldsymbol{0}$, we have
\begin{align*}
p^{u} h_{i} = h_{i}'  p^{\alpha_{i} + u} \equiv 0 \mod p^{m_i} \quad (1 \leq i \leq n,\ h_{i} = h_{i}' p^{\alpha_{i}}).
\end{align*}
From this, it follows that $u = \displaystyle \max_{1 \leq i \leq n} \{ m_{i} - \alpha_{i} \}$. 
Moreover, since $f$ is a group homomorphism and $f(\boldsymbol{h}) = \boldsymbol{x}$, we obtain
\[
p^{u} \boldsymbol{x} = p^{u} f(\boldsymbol{h}) = f(p^{u} \boldsymbol{h}) = \boldsymbol{0}.
\]
Thus,
\[
p^{u} x_{i} \equiv 0 \mod p^{m_{i}} \quad (1 \leq i \leq n)
\]
holds. Therefore, we can write $v_{p}(x_{i}) = \beta_{i} + \max \{ 0, m_{i} - u \}$, where $0 \leq \beta_{i} \leq m_{i} -  \max \{ 0, m_{i} - u \}$.
Next, we examine the conditions that $F(\boldsymbol{e}_{l})$ must satisfy. Suppose that $F(\boldsymbol{e}_{l}) = (a_{1}^{(l)},\cdots,a_{n}^{(l)})$. Since the order of $\boldsymbol{e}_{l}$ is $p^{m_{l}}$, we have
\[
p^{m_{l}} a_{i}^{(l)} \equiv 0 \mod p^{m_i}\quad (1 \leq i \leq n).
\]
As a result, $v_{p}(a_{i}^{(l)}) = \gamma_{i} + \text{max} \{ 0,m_{i} - m_{l} \}$, where $0 \leq \gamma_{i} \leq m_{i} - \text{max} \{ 0,m_{i} - m_{l} \}$. Furthermore, since $F(\boldsymbol{h}) = f(\boldsymbol{h}) = \boldsymbol{x}$, we have
\[
\boldsymbol{x} = \sum_{l=1}^{n} h_{l} F(\boldsymbol{e}_{l}).
\]
Consequently, for each $i$, we have
\begin{align*}
x_{i} \equiv \sum_{l=1}^{n} h_{l} a_{i}^{(l)} \mod p^{m_i},
\end{align*}
which implies that
\begin{align*}
x_{i}'p^{\beta_{i} + \max \{ 0, m_{i} - u \}} \equiv \sum_{l=1}^{n} h_{l}' a_{i}^{(l)'} p^{\alpha_{l} + \gamma_{l} + \text{max} \{ 0, m_{i} - m_{l} \}} \mod p^{m_{i}}
\end{align*}
must hold. Here, $x_{i} = x_{i}'p^{\beta_{i} + \max \{ 0, m_{i} - u \}}$ and $a_{i}^{(l)} = a_{i}^{(l)'} p^{\gamma_{l} + \text{max} \{ 0, m_{i} - m_{l} \}}$.
From this, it follows that for any $i$ and any $l$, there exists a $\gamma_{l}$ such that
\begin{align*}
\beta_{i} + \max \{ 0, m_{i} - u \} \geq \alpha_{l} + \gamma_{l} + \text{max} \{ 0, m_{i} - m_{l} \}
\end{align*}
is satisfied, in which case the homomorphism $f$ can be extended.

From this argument, $f$ cannot be extended if there exist some $s$ and $l$ such that, for any $\gamma_{l}$, the inequality
\begin{align*}
\beta_{s} + \max \{ 0, m_{s} - u \} < \alpha_{l} + \gamma_{l} + \text{max} \{ 0, m_{s} - m_{l} \}
\end{align*}
holds. In other words, $f$ cannot be extended if there exist some $s$ and $l$ such that
\begin{align*}
\beta_{s} + \max \{ 0, m_{s} - u \} < \alpha_{l} + \text{max} \{ 0, m_{s} - m_{l} \}.
\end{align*}
Furthermore, this is equivalent to the condition that there exists some $s$ such that
\begin{align} \label{(2.1)}
\beta_{s} + \max \{ 0, m_{s} - u \} < 
\underset{1 \leq l \leq n}{\text{min}} \{ \alpha_{l} + \text{max} \{ 0, m_{s} - m_{l} \} \}.
\end{align}
This inequality must be satisfied for any homomorphism $f \in \Hom(H, G)$, and since $\beta_{s}$ is determined by the homomorphism $f$, \Cref{(2.1)} must hold for any $\beta_{s}$. Therefore, if there exists some $s$ such that
\begin{align} \label{(2.2)}
\max \{ 0, m_{s} - u \} < 
\underset{1 \leq l \leq n}{\text{min}} \{ \alpha_{l} + \text{max} \{ 0, m_{s} - m_{l} \} \},
\end{align}
then $H$ has a homomorphism that cannot be extended. From this, we see that whether $f$ can be extended depends only on the number of times each component of the subgroup is divisible by $p$.
\end{proof}

Here, we introduce several definitions. Let $G = \displaystyle \bigoplus_{i=1}^{n} Z(p^{m_i}) \ \ ( 1 \leq m_{i} \leq m_{i+1})$. For $[\langle (p^{\alpha_{1}},\cdots,p^{\alpha_{n}}) \rangle] \in X(G) \slash {\sim}$, define the map $\Phi$ by
\[
\Phi([\langle (p^{\alpha_{1}},\cdots,p^{\alpha_{n}}) \rangle]) 
= (m_{1} - \alpha_{1},\cdots,m_{n} - \alpha_{n}) \in \Z_{\geq 0}^{n},
\]
and let $Y(G) = \text{Image}(\Phi)$. Then $\Phi:X(G) \slash {\sim} \rightarrow Y(G)$ is a bijection.
Furthermore, for $(a_{1},\cdots,a_{n})$, $(b_{1},\cdots,b_{n}) \in \Z_{\geq 0}^{n}$, define
\[
(a_{1},\cdots,a_{n}) \preceq (b_{1},\cdots,b_{n}) 
\overset{\text{def}}{\Longleftrightarrow}  
\forall i \in \{ 1,2,\cdots, n\},\ a_{i} \leq b_{i}.
\]
Additionally, for $a = (a_{1},\cdots,a_{n}) \in \Z_{\geq 0}^{n}$, define $\| a \| = \displaystyle \max_{1 \leq i \leq n} \{a_{i} \}$.
At this point, the following lemma holds.
\vskip\baselineskip

\begin{corollary} \label{cor:2.2}
Let $G = \displaystyle \bigoplus_{i=1}^{n} Z(p^{m_i}) \ \ ( 1 \leq m_{i} \leq m_{i+1})$.
If we define condition 1 as
\[
\text{condition 1}:\ ^{\exists}s \in \{ 1,2,\cdots, n\} \ s.t. \
\max \{0, m_{s} - \|\delta \| \} <
\underset{1 \leq l \leq n}{\min} \{ \max \{m_{l} - \delta_{l} ,m_{s} - \delta_{l} \} \},
\]
then
\[
Y(G) = \bigg\{ \delta = (\delta_{1},\cdots, \delta_{n}) \preceq (m_{1},\cdots,m_{n})
\Bigm\vert  \delta \text{ satisfies condition 1.} \bigg\}
\]
holds.
\end{corollary}

\begin{proof}
This follows from \Cref{(2.2)}.
\end{proof}

Since it is difficult to consider $X(G)$ and $X(G) \slash {\sim}$ under this condition, we will transform it into a more convenient form in the following \cref{lemma:2.3}. and \cref{prop:2.4}.
\vskip\baselineskip

\begin{lemma} \label{lemma:2.3}
Let $G = \displaystyle \bigoplus_{i=1}^{n} Z(p^{m_{i}})^{\lambda_{i}} \ (1 \leq m_{i} < m_{i+1},\ \lambda_{i} \in \N)$. Set $\delta = \delta_{1} \oplus \cdots \oplus \delta_{n}$ for $\delta_{i} \preceq \underbrace{(m_{i}, \cdots, m_{i})}_{\lambda_{i}\text{times}}$ . If we define condition 2 as
\[
\text{condition 2}:\ ^{\exists}s \in \{ 1, 2, \cdots, n \} \ s.t. \
\max \{ 0, m_{s} - \|\delta \| \} <
\underset{1 \leq l \leq n}{\min} \{ \max \{ m_{l} - \| \delta_{l} \|, m_{s} - \| \delta_{l} \| \} \}.
\]
Then we have
\begin{align*}
Y(G) = \bigg\{\delta = \delta_{1} \oplus \cdots \oplus \delta_{n} \ \Bigm\vert \ 
\delta \text{ satisfies condition 2.}
\bigg\}.
\end{align*}

\end{lemma}

\begin{proof}
In the context of \cref{cor:2.2}, condition 1 can be written as follows.
\[
\ ^{\exists}s \in \{ 1,2,\cdots, n\} \ s.t. \
\max \{0, m_{s} - \|\delta \| \} < 
\min_{\substack{1 \leq l \leq n\\ 1 \leq j \leq \lambda_{l}}} \{ \max \{m_{l} - \delta_{l,j} ,m_{s} - \delta_{l,j} \} \}
\]
Since the right-hand side of the inequality becomes
\begin{align*}
\min_{\substack{1 \leq l \leq n\\ 1 \leq j \leq \lambda_{l}}} \{ \max \{m_{l} - \delta_{l,j} ,m_{s} - \delta_{l,j} \} \}
&= \min_{1 \leq l \leq n} \min_{ 1 \leq j \leq \lambda_{l}} \{ \max \{m_{l} - \delta_{l,j} ,m_{s} - \delta_{l,j} \} \}\\
&= \min_{1 \leq l \leq n} \{ \max \{m_{l} - \| \delta_{l}\| ,m_{s} - \| \delta_{l} \| \} \},
\end{align*} 
$Y(G)$ is the set of $\delta$ that satisfies condition 2.
\end{proof}

Furthermore, $Y(G)$ can be rewritten as follows.
\vskip\baselineskip

\begin{proposition} \label{prop:2.4}
Consider the same setup as in \cref{lemma:2.3}. Define
\begin{align*}
f(\delta) = f_{G}(\delta) =
\begin{cases}
1 & (\| \delta \| < m_{1}) \\
\max\{ s \mid m_{s} \leq \| \delta \| \} & (\| \delta \| \geq m_{1})
\end{cases}
\end{align*}
and set condition 3 as follows:
\begin{enumerate} 
\item If $\| \delta \| < m_{1}$, then $\| \delta_{1}\| < \|\delta \|$. 
\item If $m_{f(\delta)} = \| \delta \|$, then $\| \delta_{f(\delta)}\| < m_{f(\delta)}$. 
\item If $m_{f(\delta)} < \| \delta \|$, then either $\| \delta_{f(\delta)}\| < m_{f(\delta)}$ or $\| \delta_{f(\delta)+1}\| < \| \delta \|$. 
\end{enumerate}
Then, we have
\begin{align*}
Y(G) = \bigg\{\delta = \delta_{1} \oplus \cdots \oplus \delta_{n} \Bigm\vert \delta \text{ satisfies condition 3.} \bigg\}.
\end{align*}
\end{proposition}

\begin{proof}
Let $Y'(G) = \bigg\{\delta = \delta_{1} \oplus \cdots \oplus \delta_{n} \Bigm\vert \delta \text{ satisfies condition 3.} \bigg\}$. We will prove that $Y(G) = Y'(G)$. First, we will show that $Y(G) \subset Y'(G)$. Assume that $\delta \in Y(G)$. Then,
\[
\exists s \in \{ 1,2,\cdots,n \} \ \text{s.t} \ \forall l \in \{ 1,2,\cdots,n \}, \ \max\{0, m_{s} - \|\delta \| \}<
\max \{m_{l} - \| \delta_{l}\| ,m_{s} - \| \delta_{l} \| \}.
\]\\
(A) When $\| \delta \| < m_{1}$, we will show that $\| \delta_{1}\| < \|\delta \|$.\\
Assuming $\| \delta_{1}\| = \|\delta \|$ leads to a contradiction. For $l=1$, we have
\[
m_{s} - \| \delta \| < \text{max} \{m_{1} - \| \delta_{1}\| ,m_{s} - \| \delta_{1} \| \} = m_{s} - \| \delta \|,
\]
which contradicts $\delta \in Y(G)$. Therefore, if $\delta \in Y(G)$ and $\| \delta \| < m_{1}$, then $\delta \in Y'(G)$.\\
\\
(B) When $m_{f(\delta)} = \| \delta \|$, we will show that $\| \delta_{f(\delta)}\| < m_{f(\delta)}$.\\
Assuming $\| \delta_{f(\delta)}\| = m_{f(\delta)}$ leads to a contradiction. For $l=f(\delta)$, we have
\[
\max\{0, m_{s} - \|\delta \| \}<
\max \{m_{f(\delta)} - \| \delta_{f(\delta)}\| ,m_{s} - \| \delta_{f(\delta)} \| \} = \max\{0, m_{s} - \|\delta \| \},
\]
which contradicts $\delta \in Y(G)$. Therefore, if $\delta \in Y(G)$ and $m_{f(\delta)} = \| \delta \|$, then $\delta \in Y'(G)$.\\
\\
(C) When $m_{f(\delta)} < \| \delta \|$, we will show that either $\| \delta_{f(\delta)}\| < m_{f(\delta)}$ or $\| \delta_{f(\delta)+1}\| < \| \delta \|$.\\
Assuming both $\| \delta_{f(\delta)}\| = m_{f(\delta)}$ and $\| \delta_{f(\delta)+1}\| = \| \delta \|$ leads to a contradiction. When $s \leq f(\delta)$, for $l = f(\delta)$, we have
\[
0 = \max\{0, m_{s} - \|\delta \| \}<
\max \{m_{f(\delta)} - \| \delta_{f(\delta)}\| ,m_{s} - \| \delta_{f(\delta)} \| \} = 0,
\]
which is a contradiction. When $s \geq f(\delta) +1$, for $l = f(\delta) +1$, we have
\[
m_{s} - \|\delta \| <
\max \{m_{f(\delta)+1} - \| \delta_{f(\delta)+1}\| ,m_{s} - \| \delta_{f(\delta)+1} \| \} = m_{s} - \| \delta \|,
\]
which also contradicts $\delta \in Y(G)$. Therefore, if $\delta \in Y(G)$, and either $\| \delta_{f(\delta)}\| < m_{f(\delta)}$ or $\| \delta_{f(\delta)+1}\| < \| \delta \|$, then $\delta \in Y'(G)$.

Thus, from (A), (B), and (C), we have $Y(G) \subset Y'(G)$.\\

Next, we will show that $Y(G) \supset Y(G)'$. Assume that $\delta \in Y'(G)$.\\
\\
(A') When $\| \delta \| < m_{1}$, we will show that $s = 1$ can be chosen. Since $\delta \in Y'(G)$, we have $\| \delta_{1} \| < \| \delta \|$. Therefore,
\begin{align*}
m_{1} - \| \delta \| &< m_{1} - \| \delta_{1} \|,\\
m_{1} - \| \delta \| &< m_{l} - \|\delta \| \leq  m_{l} - \|\delta_{l} \| \quad (l \in \{ 1,2,\cdots,n \}),
\end{align*}
which implies that $\delta \in Y(G)$.\\
\\
(B') When $m_{f(\delta)} = \| \delta \|$, we will show that $s = f(\delta)$ can be chosen. Since $\delta \in Y'(G)$, we have $\| \delta_{f(\delta)}\| < \| \delta \|$. Also,
\begin{align*}
\max \{m_{l} - \| \delta_{l}\| ,m_{f(\delta)} - \| \delta_{l} \| \} = 
\begin{cases}
m_{f(\delta)} - \| \delta_{l}\| &\ (l < f(\delta)), \\
m_{l} - \| \delta_{l}\| &\ (l \geq f(\delta)),
\end{cases}
\end{align*}
and for $l < f(\delta)$,
\[
\| \delta_{l} \| \leq m_{l} < \|\delta \|.
\]
Therefore,
\begin{align*}
m_{f(\delta)} - \| \delta \| &< m_{l} - \| \delta \| \leq m_{l} - \| \delta_{l}\| \quad (l > f(\delta)),\\
m_{f(\delta)} - \| \delta \| &< m_{f(\delta)} - \| \delta_{l}\| \quad (l \leq f(\delta))
\end{align*}
holds. Therefore, $\delta \in Y(G)$.\\
\\
(C') When $m_{f(\delta)} < \| \delta \|$, we will show that $s = f(\delta)+1$ can be chosen. Since $\delta \in Y'(G)$, we have either $\| \delta_{f(\delta)}\| < m_{f(\delta)}$ or $\| \delta_{f(\delta)+1}\| < \| \delta \|$. Also,
\begin{align*}
\max \{m_{l} - \| \delta_{l}\| ,m_{f(\delta)+1} - \| \delta_{l} \| \} = 
\begin{cases}
m_{f(\delta)+1} - \| \delta_{l}\| &\ (l < f(\delta)+1), \\
m_{l} - \| \delta_{l}\| &\ (l \geq f(\delta)+1),
\end{cases}
\end{align*}
and for $l \leq f(\delta)$,
\[
\|\delta_{l} \| \leq m_{l} < \|\delta \|.
\]
Thus,
\begin{align*}
m_{f(\delta)+1} - \| \delta \| &< m_{l} - \| \delta \| \leq m_{l} - \| \delta_{l}\| \quad (l > f(\delta)+1),\\
m_{f(\delta)+1} - \| \delta \| &< m_{f(\delta)+1} - \| \delta_{l}\| \quad (l \leq f(\delta)+1)
\end{align*}
holds. Furthermore, since $m_{f(\delta)+1} - \| \delta \| > 0$, we conclude that $\delta \in Y(G)$.
Thus, from (A'), (B'), and (C'), we have $Y(G) \supset Y'(G)$.
\end{proof}

\begin{lemma} \label{lemma:2.5}
Let $G = \displaystyle \bigoplus_{i=1}^{n} Z(p^{m_i}) \ \ ( 1 \leq m_{i} \leq m_{i+1})$. Let $\Proj: X(G) \rightarrow X(G)\slash {\sim}$ be the natural surjection. For $[\langle (p^{\alpha_{1}},\cdots,p^{\alpha_{n}}) \rangle] \in X(G)\slash {\sim}$, we have
\[
\# \Proj^{-1}([\langle (p^{\alpha_{1}},\cdots,p^{\alpha_{n}}) \rangle]) = \bigg( \displaystyle \prod_{i=1}^{n} \varphi (p^{\delta_{i}}) \bigg)\slash \varphi(p^{\| \delta \|}),
\]
where $\varphi$ is Euler's totient function. For $\delta \in Y(G)$, define $O_{p}(\delta)$ as
\[
O_{p}(\delta) = \# \Proj^{-1}(\Phi^{-1}(\delta)).
\]
\end{lemma}

\begin{proof}
Let $\delta_{i} = m_{i} - \alpha_{i}$ and let $l \in \underset{1\leq i \leq n }{\text{argmax}}\{ \delta_{i} \}$ be fixed. We will show that
\begin{align}
\Proj^{-1}([\langle (p^{\alpha_{1}},\cdots,p^{\alpha_{n}}) \rangle]) 
= \big\{ \langle (h_{1}p^{\alpha_{1}}, \cdots, h_{n}p^{\alpha_{n}}) \rangle \mid \ p \nmid h_{i},\ 1 \leq h_{i} < p^{\delta_{i}},\ h_{l} = 1 \big\}
\end{align}
holds. First, under the conditions $p \nmid h_{i},h_{i}',\ 1 \leq h_{i},h_{i}' < p^{\delta_{i}},\ h_{l} = h_{l}' = 1$, we will show that
\begin{align}
(h_{1},\cdots,h_{n}) \neq (h_{1}',\cdots,h_{n}') \Rightarrow \langle (h_{1}p^{\alpha_{1}}, \cdots, h_{n}p^{\alpha_{n}}) \rangle \neq \langle (h_{1}'p^{\alpha_{1}}, \cdots, h_{n}'p^{\alpha_{n}}) \rangle.
\end{align}
Since $(h_{1},\cdots,h_{n}) \neq (h_{1}',\cdots,h_{n}')$ and $h_{l} = h_{l}' = 1,\ l \in \underset{1\leq i \leq n }{\text{argmax}}\{ \delta_{i} \}$, for $1 \leq u \leq p^{\| \delta\|}$, we have
\[
u(h_{1}p^{\alpha_{1}}, \cdots, h_{n}p^{\alpha_{n}}) = (uh_{1}p^{\alpha_{1}}, \cdots, \underbrace{up^{\alpha_{l}}}_{l\text{-th}},\cdots, uh_{n}p^{\alpha_{n}}) \neq
(h_{1}'p^{\alpha_{1}}, \cdots, \underbrace{p^{\alpha_{l}}}_{l\text{-th}},\cdots, h_{n}'p^{\alpha_{n}})
\]
so (4) is shown.\\
\\
Next, we show that these elements exhaust $\Proj^{-1}([\langle (p^{\alpha_{1}},\cdots,p^{\alpha_{n}}) \rangle])$. For $p \nmid h_{i},\ l \in \underset{1\leq i \leq n }{\text{argmax}}\{ \delta_{i} \},\ 1 < h_{l} < p^{\|\delta\|}$, there exists an $h_{i}'$ such that
\[
\langle (h_{1}p^{\alpha_{1}}, \cdots, \underbrace{h_{l}p^{\alpha_{l}}}_{l\text{-th}}, \cdots, h_{n}p^{\alpha_{n}}) \rangle 
= \langle (h_{1}'p^{\alpha_{1}}, \cdots, \underbrace{p^{\alpha_{l}}}_{l\text{-th}}, \cdots, h_{n}'p^{\alpha_{n}}) \rangle.
\]
Since $p \nmid h_{l}$, by Euler's theorem, we have
\[
h_{l}^{\varphi(p^{m_{l}})} \equiv 1 \mod p^{m_{l}},
\]
so by setting $h_{i}' \equiv h_{i}^{\varphi(p^{m_{l}})} \mod p^{m_{i}}$, we can make the $l$-th component a power of $p$. Thus, (3) is shown.\\
\\
For each $i$, there are $\varphi(p^{\delta_{i}})$ choices for $h_{i}$, so we have
\[
\# \Proj^{-1}([\langle (p^{\alpha_{1}},\cdots,p^{\alpha_{n}}) \rangle]) = \bigg( \displaystyle \prod_{i=1}^{n} \varphi (p^{\delta_{i}}) \bigg)\slash \varphi(p^{\| \delta \|}).
\]
\end{proof}

\noindent\textbf{Proof of \cref{thm:1.3}} 
\begin{proof}
Let \( Y_{1}(G), Y_{2}(G), Y_{3}(G), Y_{4}(G) \) be defined as follows:
\begin{align*}
Y_{1}(G) &= \big\{ \delta \in Y(G) \mid 
1 \leq \| \delta \| < m_{1},\
\| \delta_{1} \| < \| \delta \| 
\big\} \\
Y_{2}(G) &= \bigsqcup_{i=1}^{n-1}
\big\{ \delta \in Y(G) \mid
m_{i} \leq \| \delta \| < m_{i+1},\
\| \delta_{i} \| < m_{i}
\big\} \\
Y_{3}(G) &= \bigsqcup_{i=1}^{n-1}
\big\{ \delta \in Y(G) \mid 
m_{i} < \| \delta \| < m_{i+1},\
\| \delta_{i+1} \| < \| \delta \|
\big\} \\
Y_{4}(G) &= Y_{2}(G) \cap Y_{3}(G) \\
&= \bigsqcup_{i=1}^{n-1}
\big\{ \delta \in Y(G) \mid 
m_{i} < \| \delta \| < m_{i+1},\
\| \delta_{i} \| < m_{i},\
\| \delta_{i+1} \| < \| \delta \|
\big\} 
\end{align*}
Consequently, by condition 3, we obtain
\[
Y(G) = Y_{1}(G) \sqcup (Y_{2}(G) \cup Y_{3}(G)).
\]
Thus, it follows that
\begin{align*}
\# X(G) \slash {\sim} &= \# Y(G) = \# Y_{1}(G) + \# Y_{2}(G) + \# Y_{3}(G) - \# Y_{4}(G) \\
\# X(G) &= \sum_{\delta \in Y(G)} O_{p}(\delta) = \sum_{\delta \in Y_{1}(G)} O_{p}(\delta) + \sum_{\delta \in Y_{2}(G)} O_{p}(\delta) + \sum_{\delta \in Y_{3}(G)} O_{p}(\delta) - \sum_{\delta \in Y_{4}(G)} O_{p}(\delta).
\end{align*}
(1) Let us consider \( Y_{1}(G) \).
\begin{align*}
\# Y_{1}(G) =
\begin{cases}
\displaystyle \sum_{1 \leq k < m_{1}} \sum_{\substack{\delta \in Y_{1}(G)\\ k = \|\delta \|}} 1 &\ (m_{1} > 1)\\
0 &\ (m_{1} = 1)
\end{cases}
\end{align*}
Furthermore, it holds that
\begin{align*}
\sum_{\delta \in Y_{1}(G)} O_{p}(\delta)
&= 
\begin{cases}
\displaystyle \sum_{1 \leq k < m_{1}} \sum_{\substack{\delta \in Y_{1}(G)\\ k = \|\delta \|}} O_{p}(\delta) &\ (m_{1} > 1)\\
0 &\ (m_{1} = 1)
\end{cases}
\end{align*}
When \( 1 \leq k < m_{1} \),
\begin{align*}
\sum_{\substack{\delta \in Y_{1}(G)\\ k = \|\delta \|}} 1
&= \sum_{\| \delta_{1} \| < k} \sum_{\substack{i = 2,\cdots ,n\\ \|\delta_{i}\| \leq k\\ ^{\exists} i \text{ such that } \|\delta_{i}\| = k}} 1 \\
&= k^{\lambda_{1}} \cdot \bigg( (k+1)^{|\lambda|_{n} - \lambda_{1}} - k^{|\lambda|_{n} - \lambda_{1}} \bigg)
\end{align*}

\begin{align*}
\sum_{\substack{\delta \in Y_{1}(G)\\ k = \|\delta \|}} O_{p}(\delta)
&= \sum_{\| \delta_{1} \| < k} \sum_{\substack{i = 2,\cdots ,n\\ \|\delta_{i}\| \leq k\\ ^{\exists} i \text{ such that } \|\delta_{i}\| = k}} 
\displaystyle \bigg( \prod_{\substack{1 \leq i \leq n\\ 1 \leq s_{i} \leq \lambda_{i}}} \varphi(p^{\delta_{i,s_{i}}}) \bigg) \slash \varphi(p^{k})\\
&= \bigg( \sum_{j=0}^{k-1} \varphi(p^{j})\bigg)^{\lambda_{1}} \cdot \bigg(\bigg(\sum_{j=0}^{k} \varphi(p^{j}) \bigg)^{|\lambda|_{n} - \lambda_{1}} - \bigg( \sum_{j=0}^{k-1} \varphi(p^{j}) \bigg)^{|\lambda|_{n} - \lambda_{1}} \bigg) \slash \varphi(p^{k})\\
&= p^{(k-1)\lambda_{1}} \cdot \frac{p^{k(|\lambda|_{n} - \lambda_{1})} - p^{(k-1)(|\lambda|_{n} - \lambda_{1})}}{p^{k-1}(p-1)}
\end{align*}
(2) Let us consider \( Y_{2}(G) \).
\begin{align*}
\# Y_{2}(G) =
\sum_{i=1}^{n-1} \sum_{m_{i} \leq k < m_{i+1}} \sum_{\substack{\delta \in Y_{2}(G)\\ k = \|\delta \|}} 1
\end{align*}
Furthermore, it holds that
\begin{align*}
\sum_{\delta \in Y_{2}(G)} O_{p}(\delta)
&= \sum_{i=1}^{n-1} \sum_{m_{i} \leq k < m_{i+1}} 
\sum_{\substack{\delta \in Y_{2}(G)\\ k = \|\delta \|}} O_{p}(\delta)
\end{align*}
When \( m_{i} \leq k < m_{i+1} \),
\begin{align*}
\sum_{\substack{\delta \in Y_{2}(G)\\ k = \|\delta \|}} 1 &=
\sum_{\substack{j=1,\cdots,i-1 \\ \|\delta_{j} \| \leq m_{j}}}
\sum_{\|\delta_{i} \| < m_{i}}
\sum_{\substack{l=i+1,\cdots,n \\ \|\delta_{l} \| \leq k \\ ^{\exists} l \text{ such that } \|\delta_{l} \| = k}} 1\\
&= m_{i}^{\lambda_{i}} 
\bigg( \prod_{j=1}^{i-1} (m_{j}+1)^{\lambda_{j}} \bigg)
\bigg( (k+1)^{|\lambda|_{n} - |\lambda|_{i}} - k^{|\lambda|_{n} - |\lambda|_{i}} \bigg)
\end{align*}

\begin{align*}
\sum_{\substack{\delta \in Y_{2}(G)\\ k = \|\delta \|}} O_{p}(\delta)
&= \sum_{\substack{j=1,\cdots,i-1 \\ \|\delta_{j} \| \leq m_{j}}}
\sum_{\|\delta_{i} \| < m_{i}}
\sum_{\substack{l=i+1,\cdots,n \\ \|\delta_{l} \| \leq k \\ ^{\exists} l \text{ such that } \|\delta_{l} \| = k}}
\bigg( \prod_{\substack{1 \leq i \leq n\\ 1 \leq s_{i} \leq \lambda_{i}}} \varphi(p^{\delta_{i,s_{i}}}) \bigg) \slash \varphi(p^{k})\\
&=
\bigg( \prod_{j=1}^{i-1} p^{m_{j}\lambda_{j}} \bigg) p^{(m_{i}-1)\lambda_{i}} \cdot
\frac{p^{k(|\lambda|_{n} - |\lambda|_{i})} - p^{(k-1)(|\lambda|_{n} - |\lambda|_{i})}}{p^{k-1}(p-1)}
\end{align*}
(3) Let us consider \( Y_{3}(G) \).
\begin{align*}
\# Y_{3}(G) =
\sum_{i=1}^{n-1} \sum_{m_{i} < k < m_{i+1}} \sum_{\substack{\delta \in Y_{3}(G)\\ k = \|\delta \|}} 1
\end{align*}
Furthermore, it holds that
\begin{align*}
\sum_{\delta \in Y_{3}(G)} O_{p}(\delta)
&= \sum_{i=1}^{n-1} \sum_{m_{i} < k < m_{i+1}} 
\sum_{\substack{\delta \in Y_{3}(G)\\ k = \|\delta \|}} O_{p}(\delta)
\end{align*}
When \( m_{i} < k < m_{i+1} \),
\begin{align*}
\sum_{\substack{\delta \in Y_{3}(G)\\ k = \|\delta \|}} 1
&= \sum_{\substack{j=1,\cdots,i \\ \|\delta_{j} \| \leq m_{j}}}
\sum_{\|\delta_{i+1} \| < k}
\sum_{\substack{l=i+2,\cdots,n \\ \|\delta_{l} \| \leq k \\ ^{\exists} l \text{ such that } \|\delta_{l} \| = k}} 1\\
&= k^{\lambda_{i+1}} 
\bigg( \prod_{j=1}^{i} (m_{j}+1)^{\lambda_{j}} \bigg)
\bigg( (k+1)^{|\lambda|_{n} - |\lambda|_{i+1}} - k^{|\lambda|_{n} - |\lambda|_{i+1}} \bigg)
\end{align*}

\begin{align*}
\sum_{\substack{\delta \in Y_{3}(G)\\ k = \|\delta \|}} O_{p}(\delta)
&= \sum_{\substack{j=1,\cdots,i \\ \|\delta_{j} \| \leq m_{j}}}
\sum_{\|\delta_{i+1} \| < k}
\sum_{\substack{l=i+2,\cdots,n \\ \|\delta_{l} \| \leq k \\ ^{\exists} l \text{ such that } \|\delta_{l} \| = k}} 
\bigg( \prod_{\substack{1 \leq i \leq n\\ 1 \leq s_{i} \leq \lambda_{i}}} \varphi(p^{\delta_{i,s_{i}}}) \bigg) \slash \varphi(p^{k})\\
&= \bigg( \prod_{j=1}^{i} p^{m_{j}\lambda_{j}} \bigg) 
p^{(k-1)\lambda_{i+1}} \cdot
\frac{p^{k(|\lambda|_{n} - |\lambda|_{i+1})} - p^{(k-1)(|\lambda|_{n} - |\lambda|_{i+1})}}{p^{k-1}(p-1)}
\end{align*}
(4) Let us consider \( Y_{4}(G) \).
\begin{align*}
\# Y_{4}(G) =
\sum_{i=1}^{n-1} \sum_{m_{i} < k < m_{i+1}} \sum_{\substack{\delta \in Y_{4}(G)\\ k = \|\delta \|}} 1
\end{align*}
Furthermore, it holds that
\begin{align*}
\sum_{\delta \in Y_{4}(G)} O_{p}(\delta)
&= \sum_{i=1}^{n-1} \sum_{m_{i} < k < m_{i+1}} 
\sum_{\substack{\delta \in Y_{4}(G)\\ k = \|\delta \|}} O_{p}(\delta)
\end{align*}
When \( m_{i} < k < m_{i+1} \),
\begin{align*}
\sum_{\substack{\delta \in Y_{4}(G)\\ k = \|\delta \|}} 1
&= \sum_{\substack{j=1,\cdots,i-1 \\ \|\delta_{j} \| \leq m_{j}}}
\sum_{\|\delta_{i} \| < m_{i}}
\sum_{\|\delta_{i+1} \| < k}
\sum_{\substack{l=i+2,\cdots,n \\ \|\delta_{l} \| \leq k \\ ^{\exists} l \text{ such that } \|\delta_{l} \| = k}} 1\\
&= k^{\lambda_{i+1}} m_{i}^{\lambda_{i}}
\bigg( \prod_{j=1}^{i-1} (m_{j}+1)^{\lambda_{j}} \bigg)
\bigg( (k+1)^{|\lambda|_{n} - |\lambda|_{i+1}} - k^{|\lambda|_{n} - |\lambda|_{i+1}} \bigg)
\end{align*}

\begin{align*}
\sum_{\substack{\delta \in Y_{4}(G)\\ k = \|\delta \|}} O_{p}(\delta)
&= \sum_{\substack{j=1,\cdots,i \\ \|\delta_{j} \| \leq m_{j}}}
\sum_{\|\delta_{i} \| < m_{i}}
\sum_{\|\delta_{i+1} \| < k}
\sum_{\substack{l=i+2,\cdots,n \\ \|\delta_{l} \| \leq k \\ ^{\exists} l \text{ such that } \|\delta_{l} \| = k}}
\bigg( \prod_{\substack{1 \leq i \leq n\\ 1 \leq s_{i} \leq \lambda_{i}}} \varphi(p^{\delta_{i,s_{i}}}) \bigg) \slash \varphi(p^{k})\\
&= \bigg( \prod_{j=1}^{i-1} p^{m_{j}\lambda_{j}} \bigg) 
p^{(m_{i}-1)\lambda_{i}}p^{(k-1)\lambda_{i+1}} \cdot
\frac{p^{k(|\lambda|_{n} - |\lambda|_{i+1})} - p^{(k-1)(|\lambda|_{n} - |\lambda|_{i+1})}}{p^{k-1}(p-1)}
\end{align*}
Thus, \cref{thm:1.3} has been proven.
\end{proof}

\section{Proof of Theorem 1.1}

In this section, we prove \cref{thm:1.1}.  As part of this preparation, consider the group $G = \bigoplus_{i=1}^{n} Z(m_{i}) $ with $ m_{i} \geq 1 $. By the Chinese remainder theorem, we have
\[
Z(m_{i}) = \bigoplus_{p \mid m_{i}} Z(p^{v_{p}(m_{i})}).
\]
We define $ G_{p} $ as
\[
G_{p} = \bigoplus_{i=1}^{n} Z(p^{v_{p}(m_{i})}).
\]
Then, we obtain the group isomorphism
\[
G \cong \bigoplus_{p:\text{prime}} G_{p}.
\]
Now, we will examine how $\# X(G)$ behaves for such groups $G$.

\begin{lemma} \label{lemma:3.1}
Let $G_{1}, G_{2}$ be abelian groups, and assume that the orders of $G_{1}$ and $G_{2}$ are coprime. Let $H_{1}$ and $H_{2}$ be subgroups of $G_{1}$ and $G_{2}$, respectively. Then, there exists the following isomorphism:
\[
\phi: \Hom(H_{1} \oplus H_{2}, G_{1} \oplus G_{2}) \overset{\cong}{\rightarrow} \Hom(H_{1}, G_{1}) \times \Hom(H_{2}, G_{2}).
\]
\end{lemma}

\begin{proof}
From the universality of the direct product of groups, we have
\begin{align*}
\Hom(H_{1} \oplus H_{2}, G_{1} \oplus G_{2}) 
\cong \Hom(H_{1} \oplus H_{2}, G_{1}) \times \Hom(H_{1} \oplus H_{2}, G_{2}).
\end{align*}

Furthermore, considering that $H_{1}, H_{2}, G_{1}, G_{2}$ are abelian groups and that the orders of $G_{1}$ and $G_{2}$ are coprime, we obtain

\begin{align*}
\Hom(H_{1} \oplus H_{2}, G_{1}) &\cong \Hom(H_{1}, G_{1}) \times \Hom(H_{2}, G_{1}) \cong \Hom(H_{1}, G_{1}), \\
\Hom(H_{1} \oplus H_{2}, G_{2}) &\cong \Hom(H_{1}, G_{2}) \times \Hom(H_{2}, G_{2}) \cong \Hom(H_{2}, G_{2}).
\end{align*}
Thus, \cref{lemma:3.1}. has been shown.
\end{proof}
\vskip\baselineskip

\begin{lemma} \label{lemma:3.2}
Let $G$ be a finite abelian group and denote $\Cyc(G)$ as the set of all cyclic subgroups of $G$. Let $G_{1}$ and $G_{2}$ be abelian groups with relatively a power of $p$ orders. Then, for $H_{1} \in X(G_{1})$ and $H_{2} \in \Cyc(G_{2})$, it holds that $H_{1} \oplus H_{2} \in X(G_{1} \oplus G_{2})$. Furthermore, if $H_{1} \in \Cyc(G_{1}) \backslash X(G_{1})$ and $H_{2} \in \Cyc(G_{2}) \backslash X(G_{2})$, then $H_{1} \oplus H_{2} \in \Cyc(G_{1} \oplus G_{2}) \backslash X(G_{1} \oplus G_{2})$ also holds.
\end{lemma}

\begin{proof}
Let $H_{1} \in X(G_{1})$ and $H_{2} \in \Cyc(G_{2})$. Then, since $H_{1} \in X(G_{1})$, we have
$$
^{\exists} f_{1} \in \Hom(H_{1},G_{1}) \text{ such that } ^{\forall} F_{1} \in \End(G_{1}),\ F_{1} \mid _{H_{1}} \neq f_{1}.
$$
By \cref{lemma:3.1}, any homomorphism $f \in \Hom(H_{1} \oplus H_{2},G_{1} \oplus G_{2})$ can be expressed using $f_{1} \in \Hom(H_{1},G_{1})$ and $f_{2} \in \Hom(H_{2},G_{2})$ as $f = f_{1} \oplus f_{2}$. Therefore, for the chosen $f_{1}$, the homomorphism $f_{1} \oplus f_{2} \in \Hom(H_{1} \oplus H_{2},G_{1} \oplus G_{2})$ cannot be extended to any endomorphism $F \in \End(G_{1} \oplus G_{2})$. Thus, it follows that $H_{1} \oplus H_{2} \in X(G_{1} \oplus G_{2})$. The same argument can be applied to show the claim for the second part.
\end{proof}
\vskip\baselineskip

\begin{theorem}
Let $G = \displaystyle \bigoplus_{i=1}^{n} Z(m_{i}) \cong \bigoplus_{p:prime} G_{p} \ (m_{i} \geq 1)$. Let $c(p) = \# \Cyc(G_{p})$ and $V = \{ p:prime \mid p \text{ divides } m_{1}m_{2}\cdots m_{n} \}$. Then, the following holds:
\begin{align*}
\# X(G) &= \sum_{k=1}^{n} (-1)^{k-1} \sum_{\substack{J \subset V\\ \abs{J}=k}} 
\bigg( \prod_{p \in J} \# X(G_{p}) \bigg)
\bigg( \prod_{p \notin J} c(p) \bigg).
\end{align*}
\end{theorem}

\begin{proof}
This follows from the principle of inclusion-exclusion. Tóth~\cite{Toth} provides a description regarding $c(p)$.
\end{proof}
\vskip\baselineskip

\begin{lemma} \label{lemma:3.4}
Let $G = \displaystyle \bigoplus_{i=1}^{n} Z(p^{m_i})^{\lambda_{i}} \quad (1 \leq m_{i} < m_{i+1})$ be a group that is not homocyclic. Then we have $X(G) \neq \emptyset$.
\end{lemma}

\begin{proof}
Since $G$ is not homocyclic, we have $n \geq 2$. By \cref{thm:1.3}, we have
\begin{align*}
\# X(G) &\geq T_{2}(G)\\
&\geq p^{(m_{1}-1)\lambda_{1}} \cdot \frac{p^{m_{1}(|\lambda|_{n} - |\lambda|_{1})} - p^{(m_{1}-1)(|\lambda|_{n} - |\lambda|_{1})}}{p^{m_{1}-1}(p-1)} \neq 0.
\end{align*}
\end{proof}

Let $G = \displaystyle \bigoplus_{i=1}^{n} Z(m_{i}) \cong \bigoplus_{p:\text{prime}} G_{p}$ be a group that is not the product of homocyclic groups. By \cref{lemma:3.4}, there exists a prime $p$ such that $X(G_{p}) \neq \emptyset$. Hence, there exists a cyclic subgroup $H_{p} \in X(G_{p})$. 
Furthermore, by \cref{lemma:3.2}, for any prime $q \neq p$, and for any cyclic subgroup $H_{q} \in \Cyc(G_{q})$, we have $H_{p} \oplus \bigoplus_{q \neq p} H_{q} \in X(G)$. Thus, $X(G) \neq \emptyset$, which means that a finite abelian group $G$ that is not homocyclic is not a cyclic-quasi-injective group.
This completes the proof of \cref{thm:1.1}.

\section{Application}

In this section, we will show that $ \# X(G) \slash {\sim}$ corresponds to the sum of the maximum jumps in the associated permutations by appropriately choosing $G$. Although Blecher et al.~\cite{BB} has already known the second equality in \Cref{(4.1)}, we will present another proof.

\begin{proposition} \label{prop:4.1}
Let $G = \displaystyle \bigoplus_{i=1}^{n} Z(p^{i})$ and let $S_{n}$ be the symmetric group of degree $n$. Then, we have
\begin{align} \label{(4.1)}
\# X(G) \slash {\sim} &= \sum_{k=1}^{n-1} k \cdot k! \cdot \left( (k+1)^{n-k} - k^{n-k} \right) = \sum_{\sigma \in S_{n}} \max_{1 \leq i \leq n} \{ \sigma(i) - i \}.
\end{align}
\end{proposition}

To show this, we will use two lemmas in the following.
\vskip\baselineskip

\begin{lemma} \label{lemma:4.2}
Let $W_{n} = \{ \tau \in \Z_{\geq 0}^{n} \mid \tau \preceq (n-1,n-2,\cdots,0) \}$. Furthermore, for $\sigma \in S_{n}$, we define
\[
\tau_{\sigma,i} = \# \{ j \mid i \leq j,\ \sigma(j) < \sigma(i) \}.
\]
Then, by defining $\Psi: S_{n} \rightarrow W_{n}$ as
\[
\Psi(\sigma) = \tau_{\sigma} = (\tau_{\sigma,i})_{i=1}^{n},
\]
we have that $\Psi$ is a bijection. Furthermore, for $\sigma \in S_{n}$, we have 
\[
\|\Psi(\sigma)\| = \max_{1 \leq i \leq n} \{ \sigma(i) - i \}.
\]
\end{lemma}

\begin{proof}
For each $i$, since $0 \leq \tau_{\sigma,i} \leq n-i$, it follows that $\tau_{\sigma} \in W_{n}$, which means that $\Psi$ is well-defined.

Next, we will show that $\Psi$ is injective. Let $\sigma, \sigma' \in S_{n}$ be distinct. We take the smallest index $i_{0}$ such that $\sigma(i_{0}) \neq \sigma'(i_{0})$. Without loss of generality, we can assume that $\sigma(i_{0}) < \sigma'(i_{0})$. For $i < i_{0}$, since $\sigma(i) = \sigma'(i)$, we have
\begin{align} \label{(4.2)}
\{ \sigma(i_{0}), \cdots, \sigma(n) \} = \{ \sigma'(i_{0}), \cdots, \sigma'(n) \}
\end{align}
From this, we obtain
\begin{align*}
\{ \sigma(j) \mid i_{0} \leq j , \sigma(j) < \sigma(i_{0}) \}
&= \{ \sigma'(j) \mid i_{0} \leq j, \sigma'(j) < \sigma(i_{0}) \} \\
&\subset \{ \sigma'(j) \mid i_{0} \leq j, \sigma'(j) < \sigma'(i_{0}) \}
\end{align*}
From \Cref{(4.2)}, since $\sigma(i_{0}) \in \{ \sigma'(i_{0}), \cdots, \sigma'(n) \}$, using $\sigma(i_{0}) < \sigma'(i_{0})$, we have
\[
\sigma(i_{0}) \in \{ \sigma'(j) \mid i_{0} \leq j, \sigma'(j) < \sigma'(i_{0}) \}.
\]
On the other hand, we also have 
\[
\sigma(i_{0}) \notin \{ \sigma(j) \mid i_{0} \leq j, \sigma(j) < \sigma(i_{0}) \},
\]
which implies
\[
\{ \sigma(j) \mid i_{0} \leq j, \sigma(j) < \sigma(i_{0}) \} \subsetneq
\{ \sigma'(j) \mid i_{0} \leq j, \sigma'(j) < \sigma'(i_{0}) \}.
\]
Therefore, we have
\begin{align*}
\tau_{\sigma,i_{0}} 
&= \# \{ j \mid i_{0} \leq j, \sigma(j) < \sigma(i_{0}) \} \\
&= \# \{ \sigma(j) \mid i_{0} \leq j, \sigma(j) < \sigma(i_{0}) \} \\
&< \# \{ \sigma'(j) \mid i_{0} \leq j, \sigma'(j) < \sigma'(i_{0}) \} \\
&= \# \{ j \mid i_{0} \leq j, \sigma'(j) < \sigma'(i_{0}) \} = \tau_{\sigma',i_{0}}.
\end{align*}
Thus, $\Psi$ is injective.

Now we will show that $\Psi$ is surjective. Since $\Psi$ is injective, we have $\# \Psi(S_{n}) = \# S_{n} = n!$. Since $\# W_{n} = n!$, it follows that $\# \Psi(S_{n}) = \# W_{n}$, which implies that $\Psi$ is surjective.

Let $\| \tau_{\sigma} \| = \tau_{\sigma,i}$ and define $i_{0}$ as the maximum $i$ for which $\| \tau_{\sigma} \| = \tau_{\sigma,i}$. We want to show that
\[
\| \tau_{\sigma} \| = \displaystyle \max_{1 \leq i \leq n}\{\sigma(i) - i\} = \sigma(i_{0}) - i_{0}.
\]
First, let us prove that $\sigma(i) < \sigma(i_{0})$ for any $i < i_{0}$. Suppose that $\sigma(i) > \sigma(i_{0})$, then
\begin{align*}
\tau_{\sigma,i} &= \# \{ j \mid i \leq j, \sigma(j) < \sigma(i) \}\\
&> \# \{ j \mid i_{0} \leq j, \sigma(j) < \sigma(i_{0}) \} = \tau_{\sigma,i_{0}},
\end{align*}
which contradicts the fact that $\| \tau_{\sigma} \| = \tau_{\sigma,i_{0}}$. Thus, we have $\sigma(i) < \sigma(i_{0})$ for any $i < i_{0}$.
Next, it follows that
\[
\tau_{\sigma,i_{0}} = n - (n-\sigma(i_{0})) - (i_{0}-1) -1 = \sigma(i_{0}) - i_{0}.
\]
Finally, we need to show that if $i > i_{0}$, then $\sigma(i_{0}) - i_{0} > \sigma(i) - i$, and if $i < i_{0}$, then $\sigma(i_{0}) - i_{0} \geq \sigma(i) - i$. We will only prove the former case, since the same method applies to the later case. For $i > i_{0}$, the number of elements in the set $\{ \sigma(i+1), \cdots, \sigma(n) \}$ that are greater than $\sigma(i)$ is at most $n-\sigma(i)$. Therefore, by the choice of $i_{0}$,
\begin{align*}
\sigma(i_{0}) - i_{0} &= \tau_{\sigma, i_{0}} > \tau_{\sigma, i}\\
&> (n-i) - (n-\sigma(i)) = \sigma(i) - i,
\end{align*}
proving the claim. Hence, \cref{lemma:4.2} has been shown.
\end{proof}
\vskip\baselineskip

\begin{lemma}
Let $W'_{n} = \{ \tau \in \Z_{\geq 0}^{n} \mid \tau \preceq (0,1,\cdots,n-1) \}$ and define $\rho:W'_{n} \rightarrow W_{n}$ by $\rho((\tau_{1},\tau_{2},\cdots,\tau_{n})) = (\tau_{n},\tau_{n-1},\cdots,\tau_{1})$. 
Define $\Omega:Y(G) \rightarrow W'_{n} \backslash \{ (0,\cdots,0) \}$ as
\[
\Omega(\delta) = (0,\delta_{1},\cdots, \hat{\delta}_{\| \delta \|},\cdots, \delta_{n}),
\]
then $\Omega$ is surjective. Here, $\hat{\delta}_{k}$ means excluding $\delta_{k}$. Furthermore, for any $\tau \in W'_{n} \backslash \{ (0,\cdots,0)\}$,
\[
\# \Omega^{-1}(\tau) = \| \tau \|
\]
holds.
\end{lemma}

\begin{proof}
Applying \cref{prop:2.4} to $G$, for $\delta \in Y(G)$, it holds that
\[
m_{\| \delta \|} = \| \delta \| \text{ implies } \delta_{\| \delta\|} < \| \delta \|.
\]
Thus, 
\[
\Omega^{-1}(\tau) = \{ (\tau_{2},\cdots,\tau_{\| \tau \|}, k ,\tau_{\| \tau \|+1},\cdots ,\tau_{n}) \mid 0 \leq k < \| \tau \| \}
\]
holds, which means that $\# \Omega^{-1}(\tau) = \| \tau \|$.
\end{proof}
\vskip\baselineskip

From \cref{lemma:4.2}, it follows that
\begin{align*}
\# Y(G) &= \# \bigsqcup_{\tau' \in W'_{n} \backslash \{ (0,\cdots,0)\}} \Omega^{-1}(\tau') = \sum_{\tau' \in W'_{n} \backslash \{ (0,\cdots,0)\}} \# \Omega^{-1}(\tau') = \sum_{\tau' \in W'_{n} } \| \tau' \|\\
&= \sum_{\rho^{-1}(\tau') \in W_{n}} \| \rho^{-1}(\tau') \|
= \sum_{\tau \in W_{n}} \| \tau \| = \sum_{\sigma \in S_{n}} \max_{1 \leq i \leq n} \{ \sigma(i) - i\}.
\end{align*}

In \cref{thm:1.3}, setting $m_{i} = i$ and $\lambda_{i} = 1$, we have
\begin{align*}
S_{1}(G) &= S_{3}(G) = 0 \\
S_{2}(G) &= \sum_{i=1}^{n-1} i \cdot \bigg( \prod_{j=1}^{i-1} (j+1) \bigg) \bigg( (i+1)^{n-i} - i^{n-i} \bigg) \\
&= \sum_{k=1}^{n-1} k \cdot k! \cdot \big( (k+1)^{n-k} - k^{n-k} \big).
\end{align*}
Then, we have
\begin{align*}
\# Y(G) = S_{1}(G) + S_{2}(G) + S_{3}(G) = \sum_{k=1}^{n-1} k \cdot k! \cdot \big( (k+1)^{n-k} - k^{n-k} \big).
\end{align*}
Thus, \cref{prop:4.1} has been shown.

\section*{Acknowledgement}
The author would like to thank Katsumi Kina for reviewing the manuscript and providing helpful suggestions and proofreading. The author is also grateful to Taichi Hosotani for his valuable advice during the preparation of this paper, and to Professors Ochiai, Gon, and Shirai for their insightful guidance.


\end{document}